\documentclass[12pt,a4paper]{amsart}
\pdfoutput=1
\usepackage{amsfonts}
\usepackage{amsthm}
\usepackage{amsmath}
\usepackage[normalem]{ulem}
\usepackage{amscd}
\usepackage[latin2]{inputenc}
\usepackage{t1enc}
\usepackage[mathscr]{eucal}
\usepackage{indentfirst}
\usepackage{graphicx}
\usepackage{graphics}
\usepackage{pict2e}
\usepackage{epic}
\numberwithin{equation}{section}
\usepackage[margin=2.9cm]{geometry}
\usepackage{epstopdf} 
\usepackage{hyperref}
\usepackage{mathtools}
\usepackage{amssymb}
\usepackage{empheq}
\newcounter{cnstcnt}

\hypersetup{
    colorlinks=true,
    linkcolor=blue,
    filecolor=magenta,      
    urlcolor=cyan,
}

\usepackage[
backend=biber,
style=numeric,
sorting=nyvt,
giveninits=true
]{biblatex}
\DeclareFieldFormat[article,periodical]{number}{\bibstring{number}~#1}

\renewbibmacro*{journal+issuetitle}{%
  \usebibmacro{journal}%
  \setunit*{\addspace}%
  \iffieldundef{series}
    {}
    {\newunit
     \printfield{series}%
     \setunit{\addspace}}%
  \printfield{volume}%
  \setunit{\addspace}%
  \usebibmacro{issue+date}%
  \setunit{\addcomma\space}%
  \printfield{number}%
  \setunit{\addcolon\space}%
  \usebibmacro{issue}%
  \setunit{\addcomma\space}%
  \printfield{eid}
  \newunit}

\renewbibmacro{in:}{}

\DeclareFieldFormat{pages}{#1}
\DeclareFieldFormat[article]{number}{\bibstring{number}\addnbspace #1}

\theoremstyle{plain}
\newtheorem{Th}{Theorem}[section]
\newtheorem{Lemma}[Th]{Lemma}

\newtheorem{Prop}[Th]{Proposition}

 \theoremstyle{definition}

\newtheorem{?}[Th]{Problem}

\AtBeginBibliography{\footnotesize}

\AtEndDocument{\bigskip{\footnotesize%
  \textsc{Department of Mathematics, Northwestern University, 2033 Sheridan Road, Evanston, IL 60208} \par  }}
\title{the Miyaoka-Yau Inequality \\ on smooth minimal models}
\author{Wanxing Liu}
\begin{document}
\maketitle

\begin{abstract}
      In this short note, we offer an observation that the Miyaoka-Yau inequality holds for any compact K\"{a}hler manifold with nef canonical bundle, i.e. a smooth minimal model. It follows directly from the existence of cscK metrics in a neighborhood of the canonical class which was confirmed both by the work of Dyrefelt and Song using different approaches.
\end{abstract}

\section{Introduction}
Let $(M, \omega)$ be a compact K\"{a}hler manifold of dimension $n \geq 2$ with $c_1(K_{M}) = -c_1(M)$ nef, we will prove the following theorem.
\begin{Th} \label{main}
The Miyaoka-Yau inequality
\begin{equation} \label{eq:MY}
   \tag{MY}
    (2(n+1)c_2(M) - n c_1(M)^2) \cdot (-c_1(M))^{n-2} \geq 0
\end{equation}
holds on all compact K\"{a}hler manifold with nef canonical bundle.
\end{Th}
Historically \eqref{eq:MY} was proved under different assumptions on $c_1(K_{M}) = -c_1(M)$ (more precisely on $K_{M}$), so let us first recall the relevant assumptions. First we define
\begin{equation}
    2 \pi c_1(M) = -[\sqrt{-1} \partial \bar \partial (\mathrm{log} \mathrm{det} g)].
\end{equation}
where $g$ is the metric tensor of $\omega$. Furthermore, let $\theta$ be a representative of $-c_1(M)$ we define:
\begin{enumerate}
    \item $-c_1(M)$ is nef if for any $\sigma > 0$, there exists a smooth $\varphi_{\sigma}$ such that $\theta + \sqrt{-1} \partial \bar \partial \varphi_{\sigma} > - \sigma \omega$.
    \item $-c_1(M)$ is nef and big if $-c_1(M)$ is nef and $(-c_1(M))^{n} = \int_{M} \theta^{n} > 0$.
    \item $-c_1(M)$ is semi-positive if it contains a semi-positive representative.
\end{enumerate}
Equivalently, one can define $-c_1(M)$ to be nef if it lies on the boundary of the K\"{a}hler cone:
\begin{equation*}
    \mathcal{C}_{M} := \{ [\alpha] \in H^{1,1}(M, \mathbb R) | \text{ there exists a K\"{a}hler metric }  \omega \text{ such that } [ \omega] = [\alpha]\},
\end{equation*}
but the previous definition we give reflects the fact that nefness characterizes the positivity of $-c_1(M)$. Obviously, being semi-positive implies that $-c_1(M)$ is nef directly by definition. If $-c_1(M)$ is nef one can easily conclude that $(-c_1(M))^{n} \geq 0$, but it is not necessarily big. If $-c_1(M)$ is big and nef, then $M$ is projective. By Kawamata's base point free theorem $K_{M}$ is semi-ample. It follows that $-c_1(M)$ has a semi-positive representative as some multilple of the pullback of the Fubini Study metric through the canonical map $\Phi : M \to \mathbb P^{N}$. In conclusion, nefness is a very weak condition compared with the other two conditions. However, by a corollary of the the Abundance Conjecture, nefness is expected to imply semi-positivity. Without assuming the Abundance Conjecture, when we are dealing with a nef class which is not big, one major difficulty lies in the absence of a good representative with the right positivity property since in the definition of nefness we have varying representatives for different $\sigma$. We will call a compact K\"{a}hler manifold a smooth minimal model if $-c_1(M)$ is nef and a smooth minimal model of general type if $-c_1(M)$ is nef and big.

On compact K\"{a}hler manifolds with negative first Chern class there exists a unique K\"{a}hler Einstein metric by \cite{yau1978ricci, MR494932}. \eqref{eq:MY} was initially proved by Yau \cite{MR451180} on such manifolds, and by \cite{MR451180, Miyaoka1977} on complex surfaces with big canonical bundle. Furthermore, if the equality in \eqref{eq:MY} holds on a compact K\"{a}hler manifold of negative first Chern class, then the K\"{a}hler Einstein metric is hyperbolic, i.e. its holomorphic sectional curvature is a negative constant. \eqref{eq:MY} in the case of smooth minimal models of general type was eastablished by the work of Tsuji \cite{MR976585} (see also Song-Wang \cite{MR3470713} for some clarifications) and Zhang \cite{MR2497488}. In addition, \cite{1611.05981, MR4061021} confirmed \eqref{eq:MY} for all minimal projective varieties. More recently, Nomura \cite{1802.05425} was able to obtain it under the assumption that $-c_1(M) = c_1(K_{M})$ is semi-positive but not big using K\"{a}hler Ricci flow. \cite{1611.05981} also includes a more thorough account of the historical development of \eqref{eq:MY}. See also Zhang \cite{1803.06093} for Miyaoka-Yau type inequalities on compact K\"{a}hler manifolds of almost nonpositive holomorphic sectional curvature (which implies nefness).

Our goal here is to prove \eqref{eq:MY} with no further assumption except that $-c_1(M)$ is nef. Our approach treats the case where $-c_1(M)$ is big and not big simultaneously. We will show that Theorem \ref{main} is a direct consequence of the existence of the cscK metrics in a neighborhood of the canonical class.

\begin{Th} \label{existence}
Let $(M, \omega_0)$ be a compact K\"{a}hler manifold. If the canonical class $-c_1(M)$ is nef, then for any $\varepsilon > 0$ small enough, there exists a unique cscK (constant scalar curvature K\"{a}hler) metric in the K\"{a}hler class  $- 2\pi c_1(M) + \varepsilon [\omega_0]$. 
\end{Th}

This result was obtained recently by both Dyrefelt \cite{2012.07956} and Song \cite{2004.02832}. It was first established for minimal surfaces of general type by Arezzo-Pacard \cite{MR2275832}, then Jian-Shi-Song \cite{MR3981128} proved it for smooth minimal models with semi-ample canonical line bundle. 

We will be using the following notations: $(M,\omega_0)$ is a compact manifold of dimension $n$, with a fixed K\"{a}hler metric $\omega_0$. For any K\"{a}hler metric $\omega$, we denote its Ricci curvature by $\mathrm{Ric}(\omega)$ and its scalar curvature by $R(\omega)$, and we know that $\mathrm{Ric}(\omega) \in 2 \pi c_1(M)$. A K\"{a}hler metric $\omega$ is called cscK if $R(\omega) = \frac{2  \pi n c_1(M) \cdot [\omega]^{n-1}}{[\omega]^n} $.  Finally, we will denote the unique cscK metric in $-2 \pi c_1(M) + \varepsilon [\omega_0]$ by $\omega_{\varepsilon}$.

\section{Proof of Theorem \ref{main}}

 Recall the numerical dimension $v$ of $K_M$ is defined to be 
\begin{equation}
   v : = \max \{k = 0, \ldots, n|(-c_1(M))^k \cdot [\omega_0]^{n-k} \neq 0\},
\end{equation}
and notice that if $-c_1(M)$ is big and nef, then $v = n$. We will need the following calculation.
\begin{Lemma} \label{lemma 2.3}
\begin{equation}
     \lim_{\varepsilon \to 0} \frac{2 \pi n   c_1(M) \cdot [\omega_{\varepsilon}]^{n-1}}{[\omega_{\varepsilon}]^{n}} = -  v.
\end{equation}

\end{Lemma}

\begin{proof}
Let $\eta$ be a representative of $- 2 \pi c_1(M)$, we have the following elementary expansions:

\begin{equation}
    [\omega_{\varepsilon}]^{n} = \int_{M} \omega_{\varepsilon}^n  = \int_{M}  (\eta + \varepsilon \omega_0)^n = \sum_{i = 0}^{n} {n \choose i} \varepsilon^{n-i} (- 2 \pi c_1(M))^{i} \cdot [\omega_0]^{n-i}
\end{equation}

and
\begin{equation}
\begin{aligned}
&2\pi c_1(M) \cdot [\omega_{\varepsilon}]^{n - 1} \\
&= -\int_{M} \eta \wedge (\eta + \varepsilon \omega_0)^{n-1} \\
&= -\sum_{i = 0}^{n-1}{n -1 \choose i} \varepsilon^{n-i-1} (- 2 \pi c_1(M))^{i +1} \cdot [\omega_0]^{n-i-1}.
\end{aligned}
\end{equation}

Then 
\begin{equation}
    \begin{aligned}
        &\frac{2 \pi n  c_1(M)\cdot [\omega_{\varepsilon}]^{n-1}}{[\omega_{\varepsilon}]^{n}} \\
        &= - \frac{n\sum_{i = 0}^{n-1}{n-1 \choose i}\varepsilon^{n-i-1} (- 2 \pi c_1(M))^{i +1} \cdot [\omega_0]^{n-i-1} }{\sum_{i = 0}^{n} {n \choose i} \varepsilon^{n-i} (- 2 \pi c_1(M))^{i} \cdot [\omega_0]^{n-i} } \\
        &= -  \frac{n{n-1 \choose v-1} \varepsilon^{n-v} (- 2\pi c_1(M))^v \cdot [\omega_0]^{n-v} + \ldots + n\varepsilon^{n-1} (-2 \pi c_1(M)) \cdot [\omega_0]^{n-1}}{{n \choose v} \varepsilon^{n-v} (- 2 \pi c_1(M))^v \cdot [\omega_0]^{n-v} + \ldots + \varepsilon^{n}  [\omega_0]^{n}}.
    \end{aligned}
\end{equation}
\end{proof}

The proof is concluded by noticing that $n {n-1 \choose v-1} = v {n \choose v}$. We are now ready to prove Theorem \ref{main} and the proof is going to be based on the following well-known key estimate.
\begin{Prop}[see for example \cite{MR3643615} Chapter 4 and \cite{MR2497488}]
For any K\"{a}hler metric $\omega$ on a compact K\"{a}hler manifold $M$, the following holds.
    \begin{equation}
\begin{aligned}
    &(2(n+1)c_2(M) - n c_1(M)^2) \cdot ([\omega])^{n-2}\\
    &= \frac{1}{4 \pi^2 n(n-1)} \int_{M} ((n+1)|\mathring{\mathrm{Rm}(\omega)}|_{\omega}^2 - (n+2)|\mathring{\mathrm{Ric}(\omega)}|_{\omega}^2)\omega^n\\
    & \geq \frac{1}{4 \pi^2n(n-1)} \int_{M}((n+1) |\mathring{\mathrm{Rm}(\omega)}|_{\omega}^2 - (n+2)|\mathrm{Ric}(\omega) + \omega|_{\omega}^2)\omega^n,
\end{aligned}
\end{equation}
where
\begin{equation}
    \begin{aligned}
    \omega &:= \sqrt{-1}g_{i \bar j}dz_{i} \wedge d\bar{z}_j,\\
    \mathring{\mathrm{Rm}(\omega)}_{i \bar j k \bar l} &:= \mathrm{Rm}(\omega)_{i \bar j k \bar l} - \frac{R(\omega)}{n(n+1)}(g_{i \bar j}g_{k \bar l} + g_{i \bar l}g_{k \bar j}),\\
    &\mathring{\mathrm{Ric}(\omega)}:= \mathrm{Ric}(\omega) - \frac{R(\omega)}{n}\omega.
    \end{aligned}
\end{equation}
\end{Prop}

According to the above estimate, in order to prove \ref{eq:MY}, we only need to find a sequence of K\"{a}hler metrics $\{\omega_i\}$ satisfying the following:
\begin{equation}
    \begin{aligned}
    \lim_{i \to \infty}[\omega_i] = - 2 \pi c_1(M), \\
    \lim_{i \to \infty} \int_{M} |\mathrm{Ric}(\omega_i) + \omega_i|^2_{\omega_i} \omega_i^n = 0.
    \end{aligned}
\end{equation}
Previous approaches by \cite{MR2497488} and \cite{1802.05425} took advantage of the K\"{a}hler Ricci flow which was known to exist all time (see \cite{Tsuji1988, tian2006kahler}) and  converges in cohomological sense to $-2 \pi c_1(M)$ when $-c_1(M)$ is nef. The point of \cite{MR2497488} is to use the fact that the scalar curvature along the K\"{a}hler Ricci flow is uniformly bounded which was shown by \cite{MR2544732} (Song-Tian \cite{MR3506382} also showed that the scalar curvature is bounded when $-c_1(M)$ is not big), but it relies on $-c_1(M)$ being big. Assuming semi-positivity and non-bigness, \cite{1802.05425} modified Zhang's \cite{MR2497488} approach using only the uniform lower bound on the scalar curvature but crucially exploiting the fact that the volume of the manifold along the flow is exponenentially collapsing. Instead of using K\"{a}hler Ricci flow we make use of the sequence of the cscK metrics in a neighborhood of the canonical class. Using Theorem \ref{existence}, we can take a sequence of cscK metrics $ \omega_{{\varepsilon}} \in - 2 \pi c_1(M) + \varepsilon [\omega_0]$ for $\varepsilon$ small enough, and we know that
\begin{equation}
    \begin{aligned}
        & \lim_{\varepsilon \to 0} [\omega_{\varepsilon}] = -  2 \pi c_1(M), \\
        & R(\omega_{\varepsilon}) = \frac{2 \pi n c_1(M) \cdot [\omega_{\varepsilon}]^{n-1}}{[\omega_{\varepsilon}]^n} \to - v \text{ as } \varepsilon \to 0.
    \end{aligned}
\end{equation}
by Lemma \ref{lemma 2.3}. Now the main point is:

\begin{enumerate}
    \item when$-c_1(M)$ is big and nef, $\lim_{\varepsilon \to 0} R(\omega_{\varepsilon}) = -n$.
    \item when $-c_1(M)$ is nef but not big, $\lim_{\varepsilon \to 0} R(\omega_{\varepsilon}) < \infty $.
\end{enumerate}
The following lemma is well-known.
\begin{Lemma}[see for example \cite{MR3186384} Chapter 4, Lemma 4.7]\label{lemma1}

For any K\"{a}hler metric $\omega$ on a compact K\"{a}hler manifold $M$, we have
\begin{equation}
    \int_{M} |\mathrm{Ric}(\omega)|^2 \omega^n = \int_{M} R(\omega)^2 \omega^n - 4 \pi^ 2n(n-1) c_1(M)^2 \cdot [\omega]^{n-2}.
\end{equation}
\end{Lemma}

Then a straightforward computation yields:
\begin{equation}
\begin{aligned}
   & \int_{M} |\mathrm{Ric}(\omega_{\varepsilon}) +  \omega_\varepsilon|^2_{ \omega_\varepsilon}  \omega_\varepsilon^n \\
   &=   \int_{M}( 2R( \omega_\varepsilon) + n + |\mathrm{Ric}( \omega_\varepsilon)|^2_{ \omega_\varepsilon} )  \omega_\varepsilon^n   \\
  &=   \Big(\int_{M}(2R( \omega_\varepsilon) + n + R(  \omega_\varepsilon)^2) \omega_\varepsilon^n\Big) -4 \pi^{2}n(n-1) (- c_1(M))^2 \cdot [ \omega_\varepsilon]^{n-2}  \\
  &= (2R( \omega_\varepsilon) + n + R( \omega_\varepsilon)^2)[ \omega_{\varepsilon}]^n - 4\pi^{2}n(n-1) (-c_1(M))^2 \cdot [ \omega_\varepsilon]^{n-2}\\
  & \to (-2v + n + v^2)(-2 \pi c_1(M))^{n} -  n (n - 1) (-2 \pi c_1(M))^{n}
\end{aligned}
\end{equation}
as $\varepsilon \to 0$, where $v$ is the numerical dimension of $K_{M}$. For the second equality we used Lemma \ref{lemma1}, and for the third equality we used the fact that $R(\omega_{\varepsilon})$ is constant with respect to the manifold. For calculating the limit we used Lemma \ref{lemma 2.3}. When $-c_1(M)$ is not big, we have $(-c_1(M))^n = 0$, thus
\begin{equation}
    (-2v + n + v^2)(-2 \pi c_1(M))^{n} -  n (n - 1) (-2 \pi c_1(M))^{n} = 0.
\end{equation}
When $(-c_1(M))^{n} > 0$, i.e. $v = n$, we get
\begin{equation}
\begin{aligned}
    &(-2v + n + v^2)(-2 \pi c_1(M))^{n} -  n (n - 1) (-2 \pi c_1(M))^{n} \\
    &= ((n^{2} - n) - n(n - 1))(-2 \pi c_1(M))^{n} = 0.
\end{aligned}
\end{equation}

We conclude this note by remarking that in previous approaches \cite{MR2497488} and \cite{1802.05425}  where K\"{a}hler Ricci flow was used, this computation was not possible since they did not have such a strong control on how the scalar curvature behaves along the flow globally. 

\section*{Acknowledgement}
The author would like to thank his advisor Ben Weinkove for his continued support, encouragement, and many valuable comments on the manuscript.

\medskip
\printbibliography
\end{document}